\documentclass{elsarticle}
\usepackage{graphicx}
\usepackage{subfigure}
\usepackage{color}
\usepackage{newlfont}
\usepackage{multirow}
\usepackage{longtable} 
\usepackage{ifthen}
\usepackage{alltt}
\usepackage{enumerate}
 \usepackage[text={6.25in,8.5in},centering]{geometry} 
\newcommand{\ba}{\begin{array}}
\newcommand{\ea}{\end{array}}
\newcommand{\bae}{\begin{eqnarray}}
\newcommand{\eae}{\end{eqnarray}}
\newcommand{\bea}{\begin{eqnarray*}}
\newcommand{\eea}{\end{eqnarray*}}
\newcommand{\be}{\begin{equation}}
\newcommand{\ee}{\end{equation}}

\newcommand{\pr}{{\bf Proof}~~}

\usepackage{amssymb}
\usepackage{amsthm}
\usepackage{amsmath}

\usepackage{graphicx}
\usepackage{color}
\usepackage{lscape}
\usepackage{graphicx}
\newtheorem{theorem}{\hskip\parindent\bf Theorem}[section]
\newtheorem{lemma}{\hskip\parindent\bf Lemma}[section]
\newtheorem{proposition}{\bf Proposition}[section]

\newtheorem{corollary}{\hskip\parindent\bf Corollary}[section]

\usepackage{amsmath}
\usepackage{jtb}
  \usepackage{paralist}

  \usepackage{graphics} 

  \usepackage{epsfig} 
\usepackage{epstopdf}
\usepackage{subfigure}

\begin{document}
 \markboth{modeling on Ant-Fungus interactions}{}
\title{Mathematical Modeling on Obligate Mutualism: Interactions between leaf-cutter ants and their fungus garden}
\author{Yun Kang\footnote{Applied Sciences and Mathematics, Arizona State University, Mesa,
AZ 85212, USA. E-mail: yun.kang@asu.edu}, Michael Makiyama\footnote{Ira A Fulton Engineering, Arizona State University, Tempe, AZ 85287-1804, USA. E-mail: Michael.Makiyama@asu.edu}, Rebecca Clark\footnote{School of Life Sciences, Arizona State University, Tempe, AZ 85287-1804, USA. E-mail: rebecca.m.clark@asu.edu}, Jennifer Fewell\footnote{School of Life Sciences, Arizona State University, Tempe, AZ 85287-1804, USA. E-mail: j.fewell@asu.edu}}
\begin{abstract}
We propose a simple mathematical model by applying Michaelis-Menton equations of enzyme kinetics to study the mutualistic interaction between the leaf cutter ant and its fungus garden at the early stage of colony expansion. We derive the sufficient conditions on the extinction and coexistence of these two species. In addition, we give a region of initial condition that leads to the extinction of two species when the model has an interior attractor. Our global analysis indicates that the division of labor by workers ants and initial conditions are two important {factors} that determine whether leaf cutter ants colonies and their fungus garden survive and grow can exist or not. We validate the model by doing the comparing between model simulations and data on fungal and ant colony growth rates under laboratory conditions. We perform sensitive analysis and parameter estimation of the model based on the experimental data to gain more biological insights on the ecological interactions between leaf cutter ants and their fungus garden. Finally, we give conclusions and {discuss} potential future {work}.

\end{abstract}
\bigskip
\maketitle
\section{Introduction}


Mutualistic interactions, although ubiquitous in nature, are not well understood theoretically (Boucher 1985; Herre and Bruna 1999; Hoeksema and Schwartz 2000; Holland and DeAngelis 2002; Neuhauser and  Fargione 2004). Mathematical modeling of mutualisms that {correspond} well with natural observed population dynamics have historically been difficult to formulate (Heithaus, Culver and Beattle 1980), especially for obligate mutualisms. While there are many obligate {mutualisms} found in nature, e.g., senita cacti and senita moths, coral and zooxanthellae, plant and mycorrhizal fungi, termite and protozoa  (Holland and DeAngelis 2010), there is little work done in mathematical modeling  of this topic.  Many two-species mutualism models come from modified Lotka-Volterra equations, which have been applied to a variety of ecological interactions, such as plant-pollinator interactions (Soberon and Martinez del Rio 1981; Wells 1983; Holland and DeAngelis 2002) and legume rhizobium interactions (Vandermeer and Boucher 1978; Simms and Taylor 2002; West \emph{et al.} 2002). However, these models predict unstable population dynamics that do not match the dynamics actually observed in nature (Holland and DeAngelis 2010). Recently, Holland and DeAngelis (2010) proposed a general framework using a consumer-resource approach to model the density-dependent population dynamics of mutualism. In this article, we apply Holland and DeAngelis' approach to derive a new and simple mathematical model of the population dynamics of leaf cutter ants and their fungus garden at the early colony stage, that can be validated by the experimental data. This model is unique because it includes the division of labor within the colony and the particular colony stage; incorporation of these behavioral and life-history components make it distinct from other obligate mutualism models such as plant-pollinator interactions.

This article {focuses on a species of leaf-cutter ants; these ants are fungus farmers that} harvest leaves and use them to cultivate their own food, a type of fungus, in underground gardens. The fungus that is grown by the adults feeds the ants' larvae, and {the ants feed off of the fruit of fungus}. Thus, leaf-cutter ants and their fungus garden form an obligate mutualistic relationship, in which the increasing population of ants is due to consumption of the fungus while the increasing population of fungus is due to the agricultural services provided by the ants. The interaction between ants and fungus can be categorized {as a consumer-resource} mutualism according to Holland and DeAngelis's study (2010).

Leaf-cutter ants of the genera \emph{Atta} and \emph{Acromyrmex} are among the most prevalent herbivores of the Neotropics, consuming far more vegetation than any other group of animals with comparable taxonomic diversity (Fowler \emph{et al.} 1989; H\"olldobler and Wilson 1990; Vieira-Neto and Vasconcelos 2010). In addition to their impact on plant communities, leaf-cutter ants can alter the spatio-temporal dynamics of carbon stocks, nutrient availability, susceptibility to fire, and other ecosystem properties through transferring tons of plant biomass below ground (Costa \emph{et al.} 2008). Many leaf-cutter ant species increase in abundance following natural or anthropogenic disturbances in vegetation (reviewed in Wirth \emph{et al.} 2008). However, even though the harvesting behavior of ants is well-characterized, how this translates into the growth and production efficiency of the fungus and ant population is not well understood (Wirth 2003). A mathematical treatment of the population dynamics between leaf-cutter ants and their fungus garden can be a first approach to explore these mechanisms.

It is important to note that multiple mechanisms can potentially affect the performance of individuals {and} ultimately, regulate natural populations (Murdoch 1994), and that different mechanisms may operate at different life stages (Wilbur 1980; Jonsson and Ebenman 2001; Vieira-Neto and Vasconcelos 2010). For instance, colonies of most ant species are founded by solitary queens and must pass through a series of developmental stages before reaching maturity (reviewed in {H\"olldobler} and Wilson 1990). Mortality is typically much higher for colonies in the early stages and therefore the forces acting on these colonies are likely to differ from those affecting older colonies (Wetterer 1994; Cole 2009). Studies examining variation across different stages of colony development are relatively rare and most have focused on ontogenetic changes in {individual ant} morphology and behavior (Tschinkel 1988; Wetterer 1994) rather than on factors affecting colony survival or growth. {In this article we focus on the ergonomic growth stage starting from {when} the first brood of workers reaches the adult stage}. We study the population dynamics of leaf-cutter ants and their fungus garden during this critical stage by using a simple mathematical model to explore potential sources of colony failure and specific interactions that lead to different benefits and costs for colony growth.

The model is unique also in that it incorporates behavioral effects, by considering the role of division of labor in colony growth, in particular the allocation of workers to different tasks that {can} positively or negatively affect fungal growth. Division of labor is one of the most basic and widely studied aspects of colony behavior in social insects, and addresses how individual worker behavior integrates into colony-level task organization (H\"olldobler and Wilson 1990; {Beshers and Fewell 2001}). However, models of division of labor have not previously been integrated with effects on colony growth. In this paper, we adapted a model for incipient colony growth based on simple density-dependant ant growth and death rates coupled with a fungus growth model which can be described by a generalized Michaelis-Menton equations of enzyme kinetics model. The main {purposes} of this article are three-fold:
\begin{enumerate}
\item Model functional response/numerical response based on ecological properties of {leaf} cutter ants and their fungus garden.
\item Explore how the division of labor and initial conditions can be key factors that determine the successful colony expansion at its early stage. 
\item Validate the model with experimental data and perform parameter estimations and sensitivity analysis to understand the effects of parameters and initial condition on the model outcomes.
\end{enumerate}


The rest of this article {is} organized as follows: In section 2, we introduce the biological background of leaf-cutter ants and their fungus garden, and we formulate a simple mathematical model based on ecological assumptions that are supported by data and literature. In section 3, we perform mathematical {analyses} of the proposed model: We derive the sufficient conditions for the extinction and coexistence of the two species and give a region of initial conditions that leads to the extinction of ants and fungus with a model also containing an interior attractor (Theorem \ref{th1:extinction} and \ref{th2:localstability}). These global {analyses indicate} that the division of labor by ants and initial conditions are two important {factors} {in determining} whether leaf cutter ants and their fungus garden can coexist or not (Theorem \ref{th3:basinattraction} and Corollary \ref{co:two-attractors}). In section 4, we compare simulations to data on growth rates for laboratory leaf cutter ant colonies, perform sensitivity analysis and parameter estimations for all the parameters and initial {conditions} around the nominal value. The study suggests not only that the fit of our model to data is significantly {accurate,} but also {that} the model can provide parameter values that are difficult to measure in the experiments.  In the last section, we summarize our results and discuss future work.

\section{Biological Background and model formulation}\label{sec:model}
The life cycle of the leaf-cutter ant consists of four stages: egg, larva, pupa, and adult. Fertilized eggs produce female ants ({queens and  workers}); unfertilized eggs produce male ants. The period from egg through adult stage usually lasts from 6 to 10 weeks. The average worker leaf-cutter ant lives from 4-6 months in the lab (Howard, Henneman,Cronin, Fox and Hormig 1996), while the life span of {workers} in the field is shorter and varies by season. The queen is the largest of the ants and her life span can be 10 years or more (Smith and Read 1997). 

Incipient colonies are started when mature colonies produce new winged {females}, {queens, which will mate with males} and then travel to a new location to start a new colony. They bring a small piece of fungus with them to establish a {new garden} in their new colony. After mating, queens construct nests by excavating in the soil.  Within a day of nest initiation, the queen lays lays eggs, and uses her body's energy reserves to single-handedly raise her first brood, which consists entirely of workers. The first workers emerge 6-9 weeks after the eggs are laid (H\"olldobler and Wilson 1994). These workers enlarge the nest, feed the {queen and} larvae, tend to the fungus garden and hunt for leaves for the fungus substrate. Once the first workers emerge, the 
{queen�s} only task is to produce eggs. This transition begins the ergonomic growth phase of the colony, which may last for years until the colony matures. This growth stage is a critical time in the life of the colony, because the number of workers is low relative to the number of tasks and total work effort needed for colony growth (Brown, Bot and Hart 2006).

Our study in this section aims to model the obligate mutualism interaction between {leaf cutter ants and fungus at the incipient stage of the} colony: From the time of first brood production to about 29 weeks of age. {During these early weeks, workers must allocate effort across a number of different tasks for colonies to survive and grow. The workers of established colonies, depending on maturity, may perform 20-30 tasks (Wilson 1983). Incipient colonies perform considerably fewer; of {these}, feeding fungus to the larvae, tending the fungus garden, and collecting leaves for the fungus {substrate require} the bulk of worker activity budgets, which translate {into energy expenditure} of workers (Julian and Fewell 2004; Fewell unpublished data ). These tasks can be generalized into two categories; time spent outside the colony for collecting and processing leaves and time spent inside the colony for tending and cleaning the fungus garden and taking care of queens and larvae.  For simplicity of modeling, we quantify time expenditure for these tasks as the biomass of ant.} 

Let $A(t)$ be the total biomass of ants including workers, larvae, pupae and eggs at time $t$, where $p A$ ( $0<p<1$) is the biomass of workers and $(1-p)A$ is the biomass of the remaining ants. Let $F(t)$ be the total biomass of the fungus at time $t$. Ecological assumptions of the interaction between workers and fungus at the early stage of colony are as follows:
\begin{description}
\item \textbf{A1:}  Assume that each worker has a fixed ratio of the energy spent outside the colony to the energy spent inside the colony which is $\frac{q}{1-q}$. This assumption is equivalent to a situation where workers with a population of $q pA$ ($0<q<1$) collect leaves, and the rest of workers $(1-q) p A$ tend the fungus garden and take care of queen ants as well as larvae.
\item \textbf{A2:} The ants' population increases as the queen, larvae and adult ants feed on fungus. Thus we can assume that the numerical response function for ants is the Holling Type I function, i.e., fungus biomass $F$ multiplied by a constant number $r_a$. In addition, we assume that ants suffer from density-dependent mortality due to energy consumed by foraging for leaves and taking care of the larvae and fungus garden, which will modify population growth through density-dependent self-limitation (Holland and DeAngelis 2010). Therefore, the population dynamics of ants can be described {as} follows:
\bae\label{pa}
\frac{dA}{dt}&=&\left(r_a F-d_a A\right)A
\eae where $r_a$ is a parameter that measures the maximum growth rate of ants and $d_a$ is the mortality rate of ants. 
\item \textbf{A3:} The leaf-cutter ant mutualism is unique because the workers perform {specific} tasks to maintain the life of the fungus. The population of fungus can increase only if: there are $q pA$ workers bringing back and processing leaves for the fungus; there are $(1-q) p A$ workers taking care of the fungus garden and there is healthy fungus $F$ in the garden. Thus, fungus growth is a product of two different sets of tasks performed by workers which can be represented by the following diagram:

$$ \underbrace{q\, p\,A}_{\text{energy from workers collecting and processing leaves}}+\underbrace{(1-q) \,p \,A}_{\text{energy from workers tending fungus }}+ F\rightarrow F\,+\, \mbox{new fungus}$$
Therefore, by applying the concept of the kinetics of functional response (Real 1977) we can assume that the numerical response of fungus to ants is a {Holling} Type III function $$\frac{p^2q (1-q) A^2}{b+ p^2q (1-q) A^2}$$ where $b$ is the half-saturation constant. The population of fungus decreases due {to} the consumption by ants and its mortality. Here, we assume that the fungus suffers from density-dependent mortality due to self-limiting (Holland and DeAngelis 2010). Thus, the population dynamics of fungus can be described by follows:
\bae\label{pf}
\frac{dF}{dt}&=&\left(\frac{r_f p^2q (1-q) A^2}{b+ p^2q (1-q) A^2}-d_f F- r_a c A\right)F
\eae where $r_f$ is the maximum growth rate of {the} fungus; $c$ is the conversion rate between fungus and ants and $d_f$ is the mortality rate of {the} fungus. 
See Table \ref{tab:parameters} for the biological meanings of the completed list of parameters.
\begin{table}[ht]
\begin{center}
\caption{\upshape{Biological meanings of parameters in the system \eqref{pa}-\eqref{pf}}}\label{tab:parameters}\vspace{15pt}
\begin{tabular}{|c|l|}\hline
 Parameters&Biological Meaning\\\hline
 $r_a$ &  Maximum growth rate of ants\\\hline
 $r_f$&  Maximum growth rate of fungus\\\hline
 $c$&  Conversion rate between fungus and ants \\\hline
 $d_a$& Death rate of ants \\\hline
 $d_f$& Death rate of fungus \\\hline
 $p$&Proportion of ants that are workers\\\hline
  $q$&Proportion of workers that take care of fungus\\\hline
 $b$&Half-saturation constant\\\hline\end{tabular}
\end{center}
\end{table}

\noindent \textbf{Remark:} Holling (1959) and Murdoch (1969) have discussed the application of a type III functional response associated with learning processes of predators being able to adjust their feeding rate actively based on the quantity and density of available prey. {An ant} colony similarly changes allocation of workers to different tasks as needs change with changing colony size  (H\"olldobler and Wilson 1990). In this way the ants actively modify their time expenditure and task allocation just as predators actively modify their feeding behaviors. 
\end{description}



\section{Mathematical Analysis}
Let $a=p^2q (1-q)$ and  $r_c= c r_a$, then based on the assumptions listed in Section 2, an interaction between ants and fungus at the early stage of colony may be modeled by the following differential equation:
\bae\label{pa1}
\frac{dA}{dt}&=&\left(r_a F-d_a A\right)A\\
\label{pf1}
\frac{dF}{dt}&=&\left(\frac{r_f a A^2}{b+ a A^2}-d_f F- r_c A\right)F
\eae where $a$ can be considered as a parameter measuring the division of labor in the colony of ants and other parameters are strictly positive. Since
$$a\,=\,p^2\,q\, (1-q), \,\,p\in [0,1] \,\mbox{ and } \,\,q\in [0,1],$$ {therefore,} $$a\in [0, 0.25]$$ where $a$ achieves its maximum 0.25 when $p=1, q=1/2$. In reality, $p$ is always less than 1 since the biomass of queen and other stages of ants (i.e., larvae, pupae) other than adult ants is greater than 0. Thus, $a$ is strictly less than 0.25 in the real biological system.

\begin{lemma}\label{l:pb}The system \eqref{pa1}-\eqref{pf1} is positively invariant and bounded in $\mathbb R^2_+$. In particular, if both $A(0)>0$ and $F(0)>0$, then $A(t)>0$ and $F(t)>0$ for all $t>0$.
\end{lemma}
\begin{proof}Notice that $A(0)=0$ then $A(t)=0$ for all $t\geq 0$; $F(0)=0$ then $F(t)=0$ for all $t\geq 0$. Therefore, 
\begin{enumerate}
\item If $A(0)=0$ and $F(0)=0$, then $\left(A(t),F(t)\right)=(0,0)$ for all $t\geq 0$.
\item If $A(0)=0$ and $F(0)>0$, then
$$\frac{dF}{dt}=\left(\frac{r_f a A^2}{b+ a A^2}-d_f F- r_c A\right)F=-d_f F^2 <0.$$
Thus, $\lim_{t\rightarrow \infty} F(t)=\lim_{t\rightarrow \infty} \frac{F(0)} {1+d_f t}=0$.
\item If $F(0)=0$ and $A(0)>0$, then
$$\frac{dA}{dt}=\left(r_a F-d_a A\right)A=-d_a A^2 <0.$$
Thus, $\lim_{t\rightarrow \infty} A(t)=\lim_{t\rightarrow \infty} \frac{A(0)} {1+d_a t}=0$.
\end{enumerate}
If $A(0)> 0$ and $F(0)> 0$, then due to the continuity of the system, it is impossible for either $A(t)$ or $F(t)$ {to drop} below 0. Thus, for any $A(0)\geq 0, F(0)\geq 0$, we have $A(t)\geq 0$ and $F(t)\geq 0$ for all $t\geq 0$. Now assume $A(0)\geq 0, F(0)\geq 0$, then according to the expression of $\frac{d F(t)}{dt}$, we have
$$ \frac{dF}{dt} \ = \left(\frac{r_f a A^2}{b+ a A^2}-d_f F- r_c A\right)F \leq \left(r_f-d_f F\right)F .$$
Thus, $\limsup_{t\rightarrow\infty} F(t)\leq \frac{r_f}{d_f}$. This indicates that for any $\epsilon >0$, there exists $T$ large enough, such that 
$$F(t) < \frac{r_f}{d_f} + \epsilon \mbox{ for all } t> T.$$
Therefore, we have 
$$\frac{dA}{dt}=\left(r_a F-d_a A\right)A \leq \left(r_a ( \frac{r_f}{d_f} + \epsilon)-d_a A\right)A, \mbox{ for all } t> T.$$
Since $\epsilon$ can be {arbitrarily} small, thus $\limsup_{t\rightarrow\infty} A(t)\leq \frac{r_ar_f}{d_a d_f}$. Therefore, we have shown that the system \eqref{pa1}-\eqref{pf1} is positively invariant and bounded in $R^2_+$. More specifically, the compact set $[0, \frac{r_ar_f}{d_a d_f}]\times[0,\frac{r_f}{d_f}]$ attracts all points in $\mathbb R^2_+$.

Moreover, if both $A(0)>0$ and $F(0)>0$, then we have follows
\bea
\frac{dA}{dt}&=&\left(r_a F-d_a A\right)A\geq -d_aA^2\Rightarrow A(t)\geq\frac{A(0)} {1+d_a t}>0\\
\frac{dF}{dt}&=&\left(\frac{r_f a A^2}{b+ a A^2}-d_f F- r_c A\right)F\geq-d_f F^2\Rightarrow F(t)\geq\frac{F(0)} {1+d_f t}>0
\eea 
Therefore, if both $A(0)>0$ and $F(0)>0$, then $A(t)>0$ and $F(t)>0$ for all $t>0$.
\end{proof}

\noindent\textbf{Remark:} Lemma \ref{l:pb} indicates that the population of leaf cutter ants and their fungus is bounded due to the limited resource in nature. Let $\mathring{\mathbb R}^2_+=\{(A,F)\in \mathbb R^2_+: A>0, F>0\}$, then from Lemma \ref{l:pb}, we know that $\left(A(0),F(0)\right)\in\mathring{\mathbb R}^2_+$ implies that $\left(A(t),F(t)\right)\in\mathring{\mathbb R}^2_+$ for all $t>0$. 

\begin{proposition}\label{p1:noperiod}For any initial condition taken in $\mathbb R^2_+$, the trajectory of the system \eqref{pa1}-\eqref{pf1} is converging to an equilibrium point. 
\end{proposition}
\begin{proof}By Poincar\'e-Bendixson Theorem (Guckenheimer and Holmes 1983), the omega limit set of the system \eqref{pa1}-\eqref{pf1} is either a fixed point or a limit cycle. If there exists a function $B(A,F): \mathbb R^2_+ \rightarrow R_+$, such that
$$ \frac{\partial}{\partial A}\left[B(A,F)\left(r_a F-d_a A\right)A\right]+\frac{\partial}{\partial F}\left[B(A,F)\left(\frac{r_f a A^2}{b+ a A^2}-d_f F- r_c A\right)F\right]<0,$$
then we can use Dulac's criterion (Guckenheimer and Holmes 1983) to exclude the existence of a limit cycle for the system \eqref{pa1}-\eqref{pf1} . Let $B(A,F)=\frac{1}{AF}$. Then,
$$ \frac{\partial}{\partial A}\left[B(A,F)\left(r_a F-d_a A\right)A\right]+\frac{\partial}{\partial F}\left[B(A,F)\left(\frac{r_f a A^2}{b+ a A^2}-d_f F- r_c A\right)F\right]=-\frac{d_a}{F}-\frac{d_f}{A}<0$$ holds for any $\left(A(0),F(0)\right)\in\mathring{\mathbb R}^2_+$. Therefore, by Dulac's criterion, the system \eqref{pa1}-\eqref{pf1} has no limit cycle, i.e., any trajectory of \eqref{pa1}-\eqref{pf1} starting with a {non-negative} initial condition converges to a fixed point.
\end{proof}

\noindent\textbf{Remark:} Proposition \ref{p1:noperiod} implies that the population dynamics of leaf cutter ants and their fungus is simple in the sense that they do not have {a  limit cycle}, i.e., if time $t$ is large enough, then the population of leaf cutter ants and their fungus approach to some {fixed} point. Therefore, the short time dynamics of leaf cutter ants and their fungus garden is more important since it can give us more information on the dynamics of the interaction between ants and fungus.
\begin{proposition}\label{p2:equilibria}If $a<4b\left(\frac{r_c\,r_a+d_f\,d_a}{r_a\,r_f}\right)^2$, then the system \eqref{pa1}-\eqref{pf1} has only trivial equilibrium $(0,0)$; while if $a=4b\left(\frac{r_c\,r_a+d_f\,d_a}{r_a\,r_f}\right)^2$, then the system \eqref{pa1}-\eqref{pf1} has the only positive equilibria $$\left(A^i,F^i\right)=\left(\frac{r_f \,r_a}{2\left(r_c\,r_a+d_f\,d_a\right)},\frac{r_f \,r_a\,d_a}{2r_a\left(r_c\,r_a+d_f\,d_a\right)}\right)$$ in addition to $(0,0)$; while if $a>4b\left(\frac{r_c\,r_a+d_f\,d_a}{r_a\,r_f}\right)^2$, then the system \eqref{pa1}-\eqref{pf1} has the following two positive equilibria in addition to $(0,0)$:
\bae\label{eq:interor}
\left(A^{i1},F^{i1}\right)=\left(A^{i1}, \frac{d_a}{r_a}A^{i1}\right)&\mbox{ and }&\left(A^{i2},F^{i2}\right)=\left(A^{i2}, \frac{d_a}{r_a}A^{i2}\right)
\eae where \bea
A^{i1}&=&\frac{r_f \,r_a}{2\left(r_c\,r_a+d_f\,d_a\right)}-\sqrt{\left(\frac{r_f\, r_a}{2\left(r_c\,r_a+d_f\,d_a\right)}\right)^2-\frac{b}{a}}\\
A^{i2}&=&\frac{r_f r_a}{2\left(r_c\,r_a+d_f\,d_a\right)}+\sqrt{\left(\frac{r_f\, r_a}{2\left(r_c\,r_a+d_f\,d_a\right)}\right)^2-\frac{b}{a}}
\eea 
\end{proposition}
\begin{proof}It is easy to see that $(0,0)$ is always an equilibrium of the system \eqref{pa1}-\eqref{pf1}. The nullclines of \eqref{pa1}-\eqref{pf1} can be founded as

\bea
\frac{dA}{dt}=0\Rightarrow A=0 &\mbox{ or }& F=\frac{d_a }{r_a}A\\
\frac{dF}{dt}=0\Rightarrow F=0 &\mbox{ or }& F=\frac{r_f a A^2}{d_f\left(b+ a A^2\right)}- \frac{r_c}{d_f}A
\eea
By solving $\frac{r_f a A^2}{d_f\left(b+ a A^2\right)}- \frac{r_c}{d_f}A=\frac{d_a }{r_a}A$ for A, we have the following two cases
\begin{enumerate}
\item If $a>4b\left(\frac{r_c\,r_a+d_f\,d_a}{r_a\,r_f}\right)^2$, then by simple algebraic calculations, there are the following two positive solutions of $\frac{r_f a A^2}{d_f\left(b+ a A^2\right)}- \frac{r_c}{d_f}A=\frac{d_a }{r_a}A$:
\bea
A^{i1}&=&\frac{r_f \,r_a}{2\left(r_c\,r_a+d_f\,d_a\right)}-\sqrt{\left(\frac{r_f\, r_a}{2\left(r_c\,r_a+d_f\,d_a\right)}\right)^2-\frac{b}{a}}\\
A^{i2}&=&\frac{r_f r_a}{2\left(r_c\,r_a+d_f\,d_a\right)}+\sqrt{\left(\frac{r_f\, r_a}{2\left(r_c\,r_a+d_f\,d_a\right)}\right)^2-\frac{b}{a}}
\eea Thus, the two interior equilibria are
$$\left(A^{i1},F^{i1}\right)=\left(A^{i1}, \frac{d_a}{r_a}A^{i1}\right)\,\,\mbox{ and }\,\,\left(A^{i2},F^{i2}\right)=\left(A^{i2}, \frac{d_a}{r_a}A^{i2}\right).$$
\item If $a=4b\left(\frac{r_c\,r_a+d_f\,d_a}{r_a\,r_f}\right)^2$, then the system \eqref{pa1}-\eqref{pf1} has only one positive equilibria $\left(A^i,F^i\right)$ where
$$\left(A^i,F^i\right)=\left(\frac{r_f \,r_a}{2\left(r_c\,r_a+d_f\,d_a\right)},\frac{r_f \,r_a\,d_a}{2r_a\left(r_c\,r_a+d_f\,d_a\right)}\right)$$ 
\item If $a<4b\left(\frac{r_c\,r_a+d_f\,d_a}{r_a\,r_f}\right)^2$, then there is only one trivial equilibrium: $A=0$ and $F=0$.
\end{enumerate}
Therefore, the statement of Proposition \ref{p2:equilibria} holds.
\end{proof}

\noindent\textbf{Remark:} Recall that $a$ is a parameter measuring the division of labor of workers[j9]. Proposition \ref{p2:equilibria} implies that if $a$ is too small, i.e., the ratio of adult ants that take care of fungus to adult ants that forage for leaves, $\frac{q}{1-q}$, is too small, then the system \eqref{pa1}-\eqref{pf1}
has only trivial equilibrium point $(0,0)$. This leads to the following theorem:
\begin{theorem}\label{th1:extinction}[Extinction of Two Species] If $a<4b\left(\frac{r_c\,r_a+d_f\,d_a}{r_a\,r_f}\right)^2$ , then the system \eqref{pa1}-\eqref{pf1} has global stability at $(0,0)$.
\end{theorem}
\begin{proof}From Proposition \ref{p2:equilibria}, we know that for any initial condition taken in $\mathbb R^2_+$, the trajectory of the system \eqref{pa1}-\eqref{pf1} is converging to an equilibrium point. If  $a<4b\left(\frac{r_c\,r_a+d_f\,d_a}{r_a\,r_f}\right)^2$, then according to Proposition \ref{p2:equilibria}, the only equilibrium of the system \eqref{pa1}-\eqref{pf1} is the origin $(0,0)$. Therefore, we can conclude that  the system \eqref{pa1}-\eqref{pf1} has global stability at $(0,0)$ when $a<4b\left(\frac{r_c\,r_a+d_f\,d_a}{r_a\,r_f}\right)^2$.
\end{proof}

\noindent\textbf{Biological Implications:} Theorem \ref{th1:extinction} indicates that division of labor is an important factor determining whether the {early colony stage} of leaf cutter ants can survive or not. Recall that the proportion of ants performing a task is essentially equivalent to energy devoted to a given task. In the case that {the population of adult ants is too small,} i.e., $q$ is too small, or {the population of adult ants foraging for food is too small,} i.e., $(1-q)$ is too small, then $a<q(1-q)$ will be too small such that $a<4b\left(\frac{r_c\,r_a+d_f\,d_a}{r_a\,r_f}\right)^2$. This leads to the extinction of both ants and fungus.\\

In order to investigate the biological conditions when leaf cutter ants and their fungus can coexist, we have the following theorem:
\begin{theorem}\label{th2:localstability}[Coexistence of Two Species] If $a>4b\left(\frac{r_c\,r_a+d_f\,d_a}{r_a\,r_f}\right)^2$, then the system \eqref{pa1}-\eqref{pf1} has two positive equilibria $\left(A^{i1},F^{i1}\right)$ and $\left(A^{i2},F^{i2}\right)$ where $\left(A^{i1},F^{i1}\right)$ is always unstable and $\left(A^{i2},F^{i2}\right)$ is always locally asymptotically stable.
\end{theorem}
\begin{proof}From Proposition \ref{p2:equilibria}, we know that the system \eqref{pa1}-\eqref{pf1} has two positive equilibria $\left(A^{i1},F^{i1}\right)$ and $\left(A^{i2},F^{i2}\right)$ when $a>4b\left(\frac{r_c\,r_a+d_f\,d_a}{r_a\,r_f}\right)^2$. The local stability can be determined from the eigenvalues of its Jacobian Matrices evaluated at these equilibria.

Assume that $(A^*,F^*)$ is an equilibrium point of \eqref{pa1}-\eqref{pf1}, then its Jacobian Matrices evaluated at this equilibrium can be written as follows
\bae\label{J}
J\vert_{(A^*,F^*)}=\left[\begin{array}{cc}-d_a\,A^*&r_a\,A^*\\ \frac{d_a\,A^*}{r_a}\left(2\left(r_c+\frac{d_a\,d_f}{r_a}\right)\left(1-A^*\,\left(\frac{r_c}{r_f}+\frac{d_a\,d_f}{r_a\,r_f}\right)\right)-r_c\right)&-\frac{d_a\,d_f\,A^*}{r_a}\end{array}\right]
\eae
Then we have
$$\begin{array}{ccc}trace\left(J\vert_{(A^*,F^*)}\right)&=&-d_a\,A^*\frac{r_a+d_f}{r_a}<0\\
 det\left(J\vert_{(A^*,F^*)}\right)&=&\frac{d_a\,(A^*)^2\left(r_a\,r_c+d_a\,d_f\right)\left(2r_a\,r_c\,A^*-r_a\,r_f+2d_a\,d_f\,A^*\right)}{r_a^2\,d_f}.\end{array}$$
 This implies that  if $A^*>\frac{r_a\,r_f}{2\left(r_a\,r_c+d_a\,d_f\right)}$, then $(A^*,F^*)$ is locally asymptotically stable; while if $A^*<\frac{r_a\,r_f}{2\left(r_a\,r_c+d_a\,d_f\right)}$, then $(A^*,F^*)$ is a saddle node, i.e., unstable. Since 
 \bea
A^{i1}&=&\frac{r_f \,r_a}{2\left(r_c\,r_a+d_f\,d_a\right)}-\sqrt{\left(\frac{r_f\, r_a}{2\left(r_c\,r_a+d_f\,d_a\right)}\right)^2-\frac{b}{a}}<\frac{r_a\,r_f}{2\left(r_a\,r_c+d_a\,d_f\right)}\\
A^{i2}&=&\frac{r_f r_a}{2\left(r_c\,r_a+d_f\,d_a\right)}+\sqrt{\left(\frac{r_f\, r_a}{2\left(r_c\,r_a+d_f\,d_a\right)}\right)^2-\frac{b}{a}}>\frac{r_a\,r_f}{2\left(r_a\,r_c+d_a\,d_f\right)}
\eea 
 Therefore, $\left(A^{i1},F^{i1}\right)$ is always unstable and $\left(A^{i2},F^{i2}\right)$ is always locally asymptotically stable when $a>4b\left(\frac{r_c\,r_a+d_f\,d_a}{r_a\,r_f}\right)^2$.
 \end{proof}
\noindent\textbf{Biological Implications:} Theorem \ref{th2:localstability} implies that if allocation of workers to different tasks is in a good range, i.e., $a>4b\left(\frac{r_c\,r_a+d_f\,d_a}{r_a\,r_f}\right)^2$, then both leaf cutter ants and their fungus garden can coexist, because the system \eqref{pa1}-\eqref{pf1} has a locally asymptotically stable interior equilibrium $\left(A^{i2},F^{i2}\right)$. On the other hand, for a fixed value of $a$, if $d_a, r_c$ and $\frac{d_a}{r_a}$ are small enough, then $a>4b\left(\frac{r_c\,r_a+d_f\,d_a}{r_a\,r_f}\right)^2$ holds, thus two species can coexist. Now the more interesting question is whether relative allocation among tasks is the only factor determining whether ants and fungus can coexist. The next theorem will answer this question.  \\

\begin{theorem}\label{th3:basinattraction}[Basin of Attraction of (0,0)] The trivial equilibrium $(0,0)$ is always locally asymptotically stable if $a\neq4b\left(\frac{r_c\,r_a+d_f\,d_a}{r_a\,r_f}\right)^2$. Moreover, if $a>4b\left(\frac{r_c\,r_a+d_f\,d_a}{r_a\,r_f}\right)^2$, then the basin of attraction of $(0,0)$ contains in the region $B_{(0,0)}\setminus S_{(A^{i1},F^{i1})}$ where
$$B_{(0,0)}=\left\{(A,F)\in \mathring{\mathbb R}^2_+: \frac{r_f a A^2}{d_f\left(b+ a A^2\right)}- \frac{r_c A}{d_f}\leq F<F^{i1} \right\}$$ and
$$S_{(A^{i1},F^{i1})}=\{(A,F)\in \mathring{\mathbb R}^2_+: \lim_{t\rightarrow\infty}\left(A(t),F(t)\right)=(A^{i1},F^{i1})\}.$$
\end{theorem}
\begin{proof}If $a<4b\left(\frac{r_c\,r_a+d_f\,d_a}{r_a\,r_f}\right)^2$, then according to Theorem \ref{th1:extinction}, $(0,0)$ is global stable in $\mathbb R^2_+$, thus it is locally asymptotically stable. Now we need to consider the case that $a>4b\left(\frac{r_c\,r_a+d_f\,d_a}{r_a\,r_f}\right)^2$.

First, we claim that the region defined by 
$$\Omega_1=\left\{(A,F)\in \mathring{\mathbb R}^2_+: \frac{r_f a A^2}{d_f\left(b+ a A^2\right)}- \frac{r_c A}{d_f}\leq F \leq\frac{d_a A}{r_a} \right\}$$
is positively invariant. Assume that this is not true. Then there is some initial condition $\left(A(0), F(0)\right)$ taken in $ \Omega_1$ such that for some future time $T$ such that $\left(A(T), F(T)\right)$ is leaving $\Omega_1$.  From Proposition \ref{p2:equilibria}, we know that the system \eqref{pa1}-\eqref{pf1} has only one equilibrium point $\left(A^{i1},F^{i1}\right)$ in $\Omega_1$, thus due to the continuity of the system, there exists some time $T$ such that we have one of the following two cases:
\begin{enumerate}
\item For all $0<t<T$, $$\frac{r_f a A^2(t)}{d_f\left(b+ a A^2(t)\right)}- \frac{r_c A(t)}{d_f}< F(t) <\frac{d_a A(T)}{r_a};$$ 
at $t=T$,
$$\frac{r_f a A^2(T)}{d_f\left(b+ a A^2(T)\right)}- \frac{r_c A(T)}{d_f}= F(T) \mbox{ and } F(T) <\frac{d_a A(T)}{r_a};$$ and for some $\epsilon>0$ and $T<t<T+\epsilon$, we have
$$\frac{r_f a A^2(t)}{d_f\left(b+ a A^2(t)\right)}- \frac{r_c A(t)}{d_f}> F(t) \mbox{ and } F(t) <\frac{d_a A(T)}{r_a}.$$ 
\item For all $0<t<T$, $$\frac{r_f a A^2(t)}{d_f\left(b+ a A^2(t)\right)}- \frac{r_c A(t)}{d_f}< F(t) <\frac{d_a A(T)}{r_a};$$ 
at $t=T$,
$$\frac{r_f a A^2(T)}{d_f\left(b+ a A^2(T)\right)}- \frac{r_c A(T)}{d_f}< F(T) \mbox{ and } F(T) =\frac{d_a A(T)}{r_a};$$ and for some $\epsilon>0$ and $T<t<T+\epsilon$, we have
$$F(t) >\frac{d_a A(t)}{r_a}.$$ 
\end{enumerate}
If the first case holds, then at time $t=T$ we have
\bea
\frac{dA}{dt}\vert_{t=T}&=&\left(r_a F(T)-d_a A(T)\right)A(T)=0\\
\frac{dF}{dt}\vert_{t=T}&=&\left(\frac{r_f a A^2(T)}{b+ a A^2(T)}-d_f F(T)- r_c A(T)\right)F(T)<0
\eea 
This implies that there exists some small $\epsilon$ such that $A(t)\leq A(T), F(t)< F(T)$ for all $T<t<T+\epsilon$, which {contradicts the} conditions for the first case. Similarly, we can show it is impossible for the second case to be held. Therefore, $\Omega_1$ is positively invariant.

Now we will show that $B_{(0,0)}$ is positively invariant. Define
$$\Omega_2=B_{(0,0)}\setminus\Omega_1\,\,\mbox{ and } \Omega_3=\Omega_1\setminus\left\{\left(A^{i1},F^{i1}\right)\right\}.$$
Then, $\Omega_3$ is also positively invariant since $\left(A^{i1},F^{i1}\right)$ is an equilibrium point and $B_{(0,0)}=\Omega_2\cup\Omega_3$. For any initial condition $\left(A(0), F(0)\right)$ taken in $B_{(0,0)}$, there are the following two cases:
\begin{enumerate}
\item If $\left(A(0), F(0)\right)\in\Omega_3$, then $\left(A(t), F(t)\right)\in\Omega_3$ for all $t>0$ since $\Omega_3$ is positively invariant;
\item If $\left(A(0), F(0)\right)\in\Omega_2$, then either $\left(A(t), F(t)\right)\in\Omega_2$ for all $t>0$ or there exists some $T$ such that
\bea
\frac{dA}{dt}\vert_{t=T}&=&\left(r_a F(T)-d_a A(T)\right)A(T)=0\\
\frac{dF}{dt}\vert_{t=T}&=&\left(\frac{r_f a A^2(T)}{b+ a A^2(T)}-d_f F(T)- r_c A(T)\right)F(T)<0
\eea
This implies that $\left(A(T), F(T)\right)\in\Omega_3$. Since $\Omega_3$ is positively invariant, then $\left(A(t), F(t)\right)\in\Omega_3$ for all $t>T$.
\end{enumerate}
Therefore,  $B_{(0,0)}$ is positively invariant.

Define a Lyapunov function $V=A^\alpha F^\beta: B_{(0,0)}\rightarrow \mathbb R^2_+$ where both $\alpha$ and $\beta$ are positive. Then we have
\bea
\frac{dV}{dt}&=&\alpha A^{\alpha-1} F^\beta\frac{dA}{dt}+\beta A^\alpha F^{\beta-1}\frac{dF}{Fdt}=\alpha A^\alpha F^\beta\left(r_a F-d_a A\right)+ \beta A^\alpha F^\beta\left(\frac{r_f a A^2}{b+ a A^2}-d_f F- r_c A\right)\\&=&V\left[\alpha\left(r_a -\frac{\beta d_f}{\alpha} \right)F+ \beta\left(\frac{r_f a A^2}{b+ a A^2}- \left(r_c+\frac{\alpha d_a}{\beta}\right) A\right)\right]\eea
Choose $\alpha,\beta$ such that  $r_a -\frac{\beta d_f}{\alpha} =0$, i.e., $\frac{\beta}{\alpha}=\frac{r_a}{d_f}$,  then the expression of $\frac{dV}{dt}$ can be simplified as 
\bea
\frac{dV}{dt}&=&V\left[\beta\left(\frac{r_f a A^2}{b+ a A^2}- \left(r_c+\frac{d_f d_a}{r_a}\right) A\right)\right]=\beta V\,A\left[\frac{r_f a A-b \left(r_c+\frac{d_f d_a}{r_a}\right)- a \left(r_c+\frac{d_f d_a}{r_a}\right)A^2}{b+ a A^2}\right]\eea
Define $f(A)=r_f a A-b \left(r_c+\frac{d_f d_a}{r_a}\right)- a \left(r_c+\frac{d_f d_a}{r_a}\right)A^2$, then the sign of $\frac{dV}{dt}$ depends on the sign of $f(A)$. Notice that $f(A)=(A-A^{i1})(A^{i2}-A)$ is negative if $0<A<A^{i1}$ and any point $(A,F)\in B_{(0,0)}$ satisfying $0<A<A^{i1}, 0<F<F^{i1}$.

Since $ B_{(0,0)}$ is positively invariant, {for} any initial condition taken in $B_{(0,0)}$, we have $\frac{dV}{dt}<0$ for all future time. This indicates that 
$A(t)$ and $F(t)$ approach to some fixed point contained in $\overline{B_{(0,0)}}$ (the closure of $B_{(0,0)}$). Notice that $\overline{B_{(0,0)}}$ contains only $(0,0)$ and $(A^{i1},F^{i1})$. According to Theorem \ref{th2:localstability}, $(A^{i1},F^{i1})$ is unstable when $a>4b\left(\frac{r_c\,r_a+d_f\,d_a}{r_a\,r_f}\right)^2$. Then based on Hartman-Grobman Theorem (Robinson 1998), any point in $B_{(0,0)}\setminus S_{(A^{i1},F^{i1})}$ will not approach to $\left(A^{i1},F^{i1}\right)$, {and} therefore, it will approach to $(0,0)$.

Therefore, the statement of Theorem \ref{th3:basinattraction} holds.
\end{proof}

\noindent A direct corollary of Proposition \ref{p1:noperiod}, Theorem \ref{th2:localstability} and Theorem \ref{th3:basinattraction} is as follows:
\begin{corollary}\label{co:two-attractors}If $a>4b\left(\frac{r_c\,r_a+d_f\,d_a}{r_a\,r_f}\right)^2$, then the system \eqref{pa1}-\eqref{pf1} has two attractors 
$$ (0,0)\,\,\mbox{  and  } \left(A^{i2},F^{i2}\right).$$ If the initial condition $\left(A(0),F(0)\right)$ is too small such that it contained in $B_{(0,0)}\setminus S_{(A^{i1},F^{i1})}$ , then $$\lim_{t\rightarrow\infty}\left(A(t),F(t)\right)=(0,0);$$ while the initial condition $\left(A(0),F(0)\right)$ is large enough, then $$\lim_{t\rightarrow\infty}\left(A(t),F(t)\right)= \left(A^{i2},F^{i2}\right).$$
\end{corollary}

\noindent\textbf{Biological Implications:} Theorem \ref{th3:basinattraction} and Corollary \ref{co:two-attractors} suggests that the initial population of leaf cutter ants and fungus is another important factor that determines whether ants and fungus can coexist or not. If initial population is contained in $B_{(0,0)}\setminus S_{(A^{i1},F^{i1})}$ , then both ants and fungus will {go extinct} even if the division of labor is in a good range, i.e., $a>4b\left(\frac{r_c\,r_a+d_f\,d_a}{r_a\,r_f}\right)^2$.
\section{Numerical simulations, data and sensitive analysis}\label{sec:allee}
In this section, we validate our model \eqref{pa1}-\eqref{pf1} by performing numerical simulations, sensitivity analysis and parameter estimations based on the experiment data. The numerical simulations fit the data very well (see Figure \ref{fig:AFfit} and \ref{fig:AFration}), which suggests that our model \eqref{pa1}-\eqref{pf1} is well defined.  Sensitivity analysis around these chosen parameter values provides information on the governing factors for the ecological process modeled by \eqref{pa1}-\eqref{pf1}. Parameter estimations with different initial guessing values indicate that the population dynamics of leaf cutter ants and their fungus may be unstable at the early stage of colony expansion, as supported by the empirical data (Clark and Fewell in preparation). 
\subsection{Numerical simulations and experimental data}
In this subsection, we compare the numerical simulations of the model  \eqref{pa1}-\eqref{pf1} by using parameter values in certain intervals (see these values in Table \ref{tab2:v_parameters}). These intervals are obtained from the approximations according to data and literature (Brown, Bot and Hart  2006; Clark and Fewell in preparation). Table \ref{tab2:v_parameters} lists the range of parameters and the specific values (i.e., $r_a=0.1,r_f=.7,d_a=0.1,d_f=0.2,b=0.002,r_c=.0045, a=0.2 ,A(6)=0.05,F(6)=0.3$) that generate dashed lines in Figure \ref{fig:AFfit} and \ref{fig:AFration} from week 6 to week 29. Other values in the interval can {generate similar} dynamics as the chosen {values,} but the chosen values {fit the data well.}

\begin{table}[ht]
\begin{center}
\caption{\upshape{Intervals and chosen values of parameters in the system \eqref{pa1}-\eqref{pf1}}}\label{tab2:v_parameters}\vspace{15pt}
\begin{tabular}{|l|l|l|l|}\hline
 Parameters& Intervals&Chosen values \\\hline
 $r_a$: Maximum growth rate of ants&(0.05, 0.3)&0.1\\\hline
 $r_f$: Maximum growth rate of fungus&(0.01,1)&0.7\\\hline
 $r_c$:  Conversion rate between fungus and ants&(0.001,10)&0.0045 \\\hline
 $d_a$: Death rate of ants&(0.001,1)&0.1 \\\hline
 $d_f$: Death rate of fungus &(0.001,1)&0.2\\\hline
  $b$: Half-saturation constant&(0.001,10)&0.002\\\hline
  $a$: Measurement of the division of labor& (0, 0.25)&0.2\\\hline
    $A(6)$:Biomass of ants at week 6& (0.001, 0.1)&0.05\\\hline
      $F(6)$:Biomass of ants at week 6& (0.001, 1)&0.3\\\hline
\end{tabular}
\end{center}
\end{table}

Figure \ref{fig:AFfit} and Figure \ref{fig:AFration} provide the comparison between ecological data (solid lines with error bars) and simulations (dashed lines) generated by the model \eqref{pa1}-\eqref{pf1} when $r_a=0.1,r_f=.7,d_a=0.1,d_f=0.2,b=0.002,r_c=.0045, a=0.2 ,A(6)=0.05,F(6)=0.3$. 
\begin{enumerate}

\item In Figure \ref{fig:AFfit}, the left figure is the biomass of Ants {v.s.} Time in weeks and the right figure is the biomass of Fungus v.s. Time in weeks. By comparison, we can see that the simulations fit the data very well, especially for biomass of fungus. Overall, the simulation of the biomass of ants is larger than the experimental data (the right figure). This is expected, because the equation \eqref{pa1} models the biomass of all ants including the queen, eggs, larvae, pupae and workers, while the experiment only measures the biomass of workers. In addition, both data and simulations suggest that ants have exponential growth while the fungus has linear-like growth from weeks 6-29. Recall that our focus is the ergonomic growth stage of the ants which starts when the first workers appear. The exponential growth of ants at this growth stage confirms the study by Oster and Wilson (1978).
\begin{figure}[ht]
\begin{center}
\includegraphics[width=160mm]{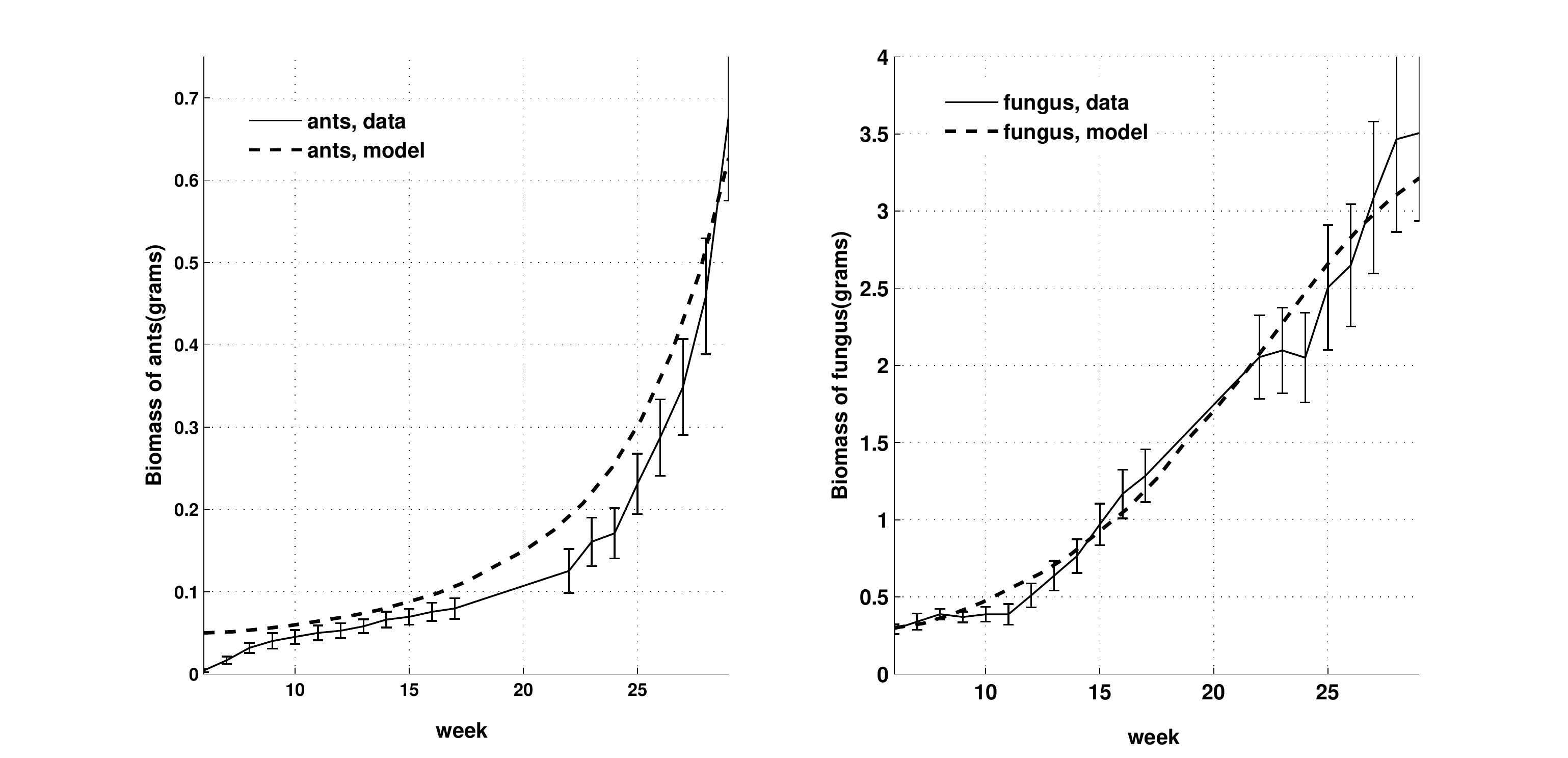}
 \caption{The solid lines with error bars are from data and the dashed lines are simulations generated from the model \eqref{pa1}-\eqref{pf1} when $r_a=0.1,r_f=.7,d_a=0.1,d_f=0.2,b=0.002,r_c=.0045, a=0.2 ,A(6)=0.05,F(6)=0.3$. The left figure is the biomass of Ants v.s Time in weeks and the right figure is the biomass of Fungus v.s. Time in weeks.}
 \label{fig:AFfit}
\end{center}
  \end{figure}
\item The right figure of Figure \ref{fig:AFration} represents  $\log_{10}(\mbox{the biomass of ants+1})$ v.s.  $\log_{10}(\mbox{the biomass of fungus+1})$, which provides the information on the relationship between the growth rate of ants and the growth rate of fungus. Simulations fit data extremely well. Both {suggest that the growth rate of ants and fungus increase over time,} and the growth rate of ants increases faster than fungus, which may be caused by changes in the efficiency of the conversion between ants and fungus at the early colony stage.

The left figure of Figure \ref{fig:AFration} is the ratio of $\log_{10}(\mbox{the biomass of ants+1})$ to $\log_{10}(\mbox{the biomass of fungus+1})$ v.s. time in weeks, which provides information on the relative growth rate of ants to fungus: the simulation fits data very well from week 10 to week 29 but shows some inconsistency between the data and the model fitting during week 6 to week 9. In this case the model is a more accurate descriptor of population dynamics than the collected data because the biomass of the ants during this time {consists} almost entirely of immature workers or ants in the larvae/pupae stage. The data do not account for this ant biomass and {thus, } from week 6 to week 9 the ant population may be largely underrepresented. This under-representation leads to an increase in the slope of data when the actual result should be closer to the model output during this time. {Thus,} the possible explanations for the inconsistency between the data and the model fitting during week 6 to week 9 can be summarized as follows: 1. The equation \eqref{pa1} models the biomass of all ants while the experiment only measures the biomass of {workers; thus,} the model \eqref{pa1}-\eqref{pf1} should generate the larger ratio of $\log_{10}(\mbox{the biomass of ants+1})$ to $\log_{10}(\mbox{the biomass of fungus+1})$; 2. For the first few weeks (week 6-9), the real population dynamics of ants and fungus are highly unstable and may have very different ecological properties than our model assumptions. Stochastity and multiple life stages of leaf cutter ants (e.g., eggs, larvae, pupae) may be considered in future models.

 \begin{figure}[ht]
 \begin{center}
   \includegraphics[width=160mm]{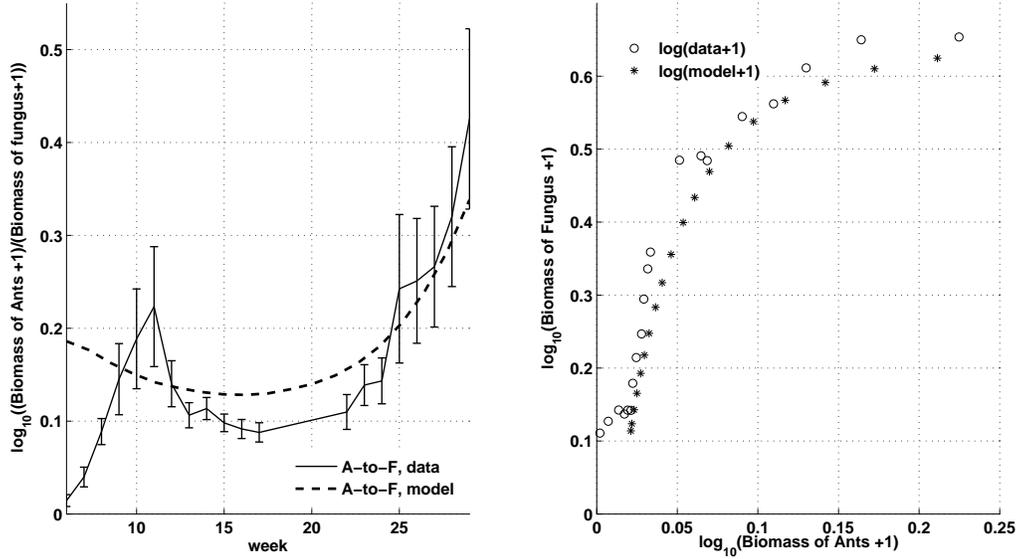}
\caption{The solid lines with error bars are from data and the dashed lines are simulations generated from the model \eqref{pa1}-\eqref{pf1} when $r_a=0.1,r_f=.7,d_a=0.1,d_f=0.2,b=0.002,r_c=.0045, a=0.2 ,A(6)=0.05,F(6)=0.3$. In the left figure, the x-axis represents time in weeks and the y-axis represents the ratio of $\log_{10}(\mbox{the biomass of ants+1})$ to $\log_{10}(\mbox{the biomass of fungus+1})$; while in the right figure, the x-axis represents represents the value of $\log_{10}(\mbox{the biomass of ants+1})$ and the y-axis represents the value of $\log_{10}(\mbox{the biomass of fungus+1})$}
 \label{fig:AFration}
\end{center}\end{figure}
\end{enumerate}

The fitting of the model to the data for the ants and fungus population is evident from {Figures} \ref{fig:AFfit} and \ref{fig:AFration}. The comparison between simulations generated by the model \eqref{pa1}-\eqref{pf1} and data suggests not only that the fit of the model to data is accurate but also that parameters match with expected values for growth, death, and division of labor. A recent experiment study by Clark and Fewell (in preparation) on leaf-cutter ants shows that the parameter values for $r_a, r_f, d_f$ that generate {Figures} \ref{fig:AFfit} and \ref{fig:AFration} are very close to actual data. The initial condition used for generating {Figures} \ref{fig:AFfit} and \ref{fig:AFration} is the mean value from the experimental data by Clark \emph{et al} (preprint). A study by Brown, Bot and Hart (2006) on mortality rates of leaf cutter ants and division of labor suggests that the death rate matches with the parameter $d_a$.  Notice that the values of parameters such as $r_c, b$ are difficult to measure in the experiments. The good fit of the model to data (see {Figures} \ref{fig:AFfit} and \ref{fig:AFration}) generated by the values listed in Table \ref{tab2:v_parameters} provides an approximation of $r_c$ and $b$. In the following subsection, we will examine the sensitivity of these parameter values and the initial condition. The parameter $d_f$ is difficult to measure experimentally because of the efficiency of the mutualistic relationship between the ants and fungus.

The leaf-cutter ants cultivate and consume the fungus inside underground gardens, so data collection in the field ultimately results in the death of the colony. Our data measures of fungal biomass include estimates of fungal loss from waste piles; however, {calculating plant waste underestimates the death rate of the fungus}, because much of the fungus is consumed or used for other purposes inside the {colony. By the time} the fungus biomass has been disposed of, it is only a fraction of its original {biomass}, making any measurement of its mass/area inaccurate. {Additionally, actual consumption of the fungus cannot be easily observed in a field or laboratory setting}.

\subsection{Sensitivity analysis}
 Input factors for our mathematical model \eqref{pa1}-\eqref{pf1} consist of seven parameters and two initial conditions for independent and dependent variables of the model. Because of natural variation, error in measurements, or simply a lack of current techniques to measure some parameters, it is necessary to perform sensitivity analysis to identify critical inputs (parameters and initial conditions) of our model and quantifying how input uncertainty impacts model outcomes (i.e., the dynamics of the ants and fungus biomass of ants $A(t), F(t)$).  In this subsection, sensitivity measure of the model \eqref{pa1}-\eqref{pf1}  is computed numerically by performing multiple simulations varying input factors around the nominal value listed in Table \ref{tab2:v_parameters}.

Sensitivity analysis for parameters and initial conditions were performed using an extension of the MATLAB function ODE23tb, a stiff solver for ordinary differential equations. The extension maintains the same calling sequences as ODE23tb; it differs only in the algorithm used to compute the derivatives.
The algorithms used in this instance are the internal numerical differentiation and iterative approximation based on directional derivatives methods described by H.G. Bock (1981) and T. Maly with L. R. Petzold (1996) respectively. The output of the function is similar to that of ODE23tb with an additional array containing the derivatives (sensitivities) of the solution with respect to a given parameter vector. The sensitivity of all parameters and initial conditions around $r_a=0.1,r_f=.7,d_a=0.1,d_f=0.2,b=0.002,r_c=.0045, a=0.2 ,A(6)=0.05,F(6)=0.3$ from week 6 to week 29 are shown in Figure \ref{fig:ra_rf}-\ref{fig:A0_F0}. 
 \begin{figure}[ht]
 \begin{center}
   \includegraphics[width=160mm]{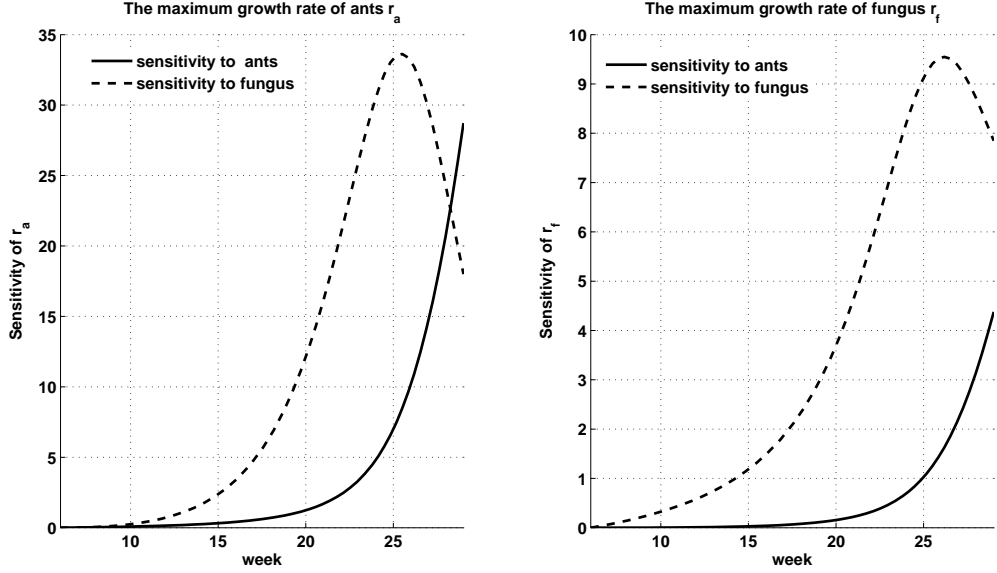}
\caption{The left figure is the sensitivity of the maximum growth rate of ants $r_a$ and the right figure is the sensitivity of the maximum growth rate of fungus $r_f$ for the model \eqref{pa1}-\eqref{pf1} around the chosen values where $r_a=0.1,r_f=.7,d_a=0.1,d_f=0.2,b=0.002,r_c=.0045, a=0.2 ,A(6)=0.05,F(6)=0.3$.}
 \label{fig:ra_rf}
\end{center}\end{figure}

\begin{figure}[ht]
 \begin{center}
   \includegraphics[width=160mm]{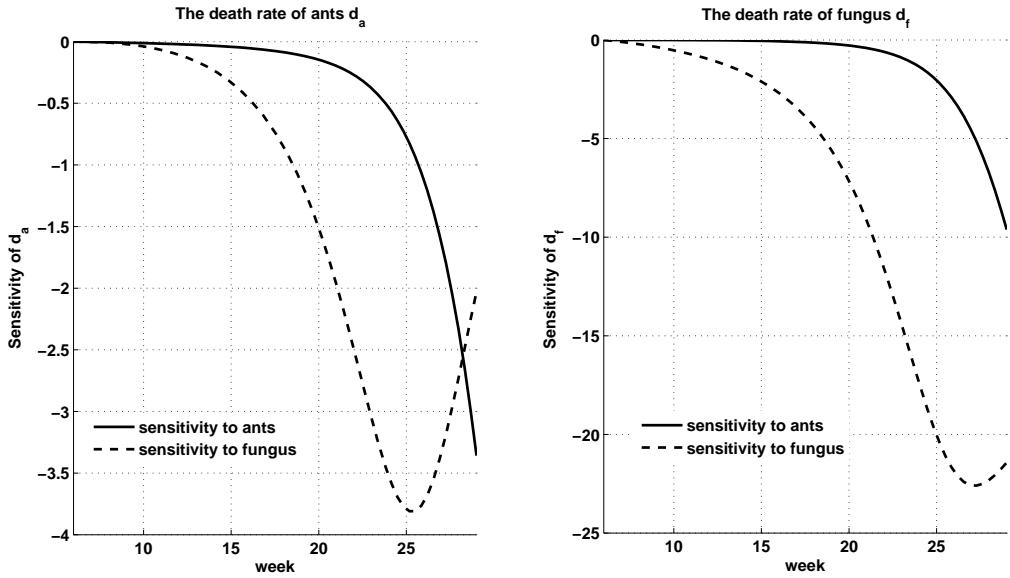}
\caption{The left figure is the sensitivity of the death rate of ants $d_a$ and the right figure is the sensitivity of the death rate of fungus $d_f$ for the model \eqref{pa1}-\eqref{pf1} around the chosen values where $r_a=0.1,r_f=.7,d_a=0.1,d_f=0.2,b=0.002,r_c=.0045, a=0.2 ,A(6)=0.05,F(6)=0.3$.}
 \label{fig:da_df}
\end{center}\end{figure}
\begin{figure}[ht]
 \begin{center}
   \includegraphics[width=160mm]{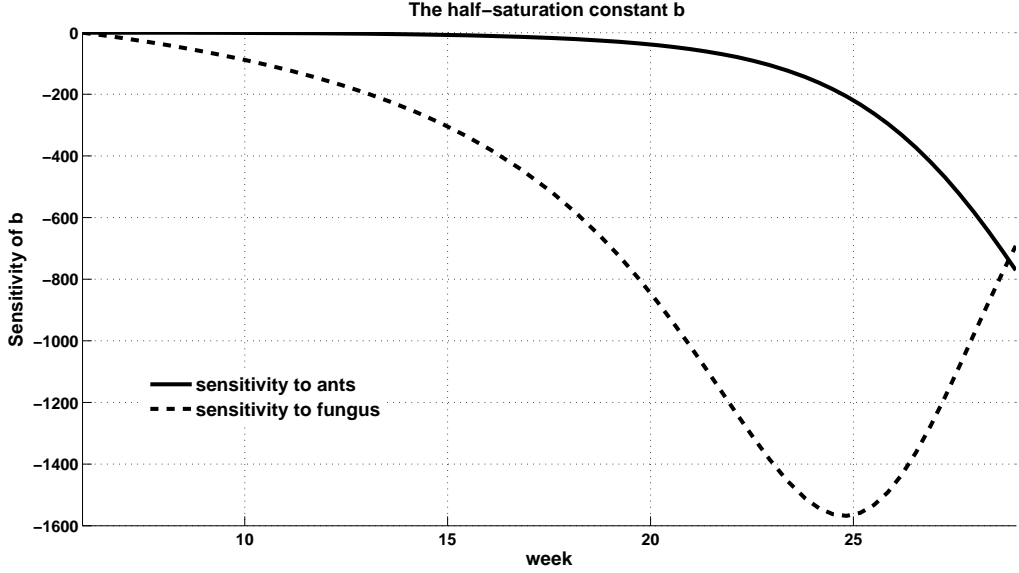}
\caption{The sensitivity of the half-saturation constant $b$ for the model \eqref{pa1}-\eqref{pf1} around the chosen values where $r_a=0.1,r_f=.7,d_a=0.1,d_f=0.2,b=0.002,r_c=.0045, a=0.2 ,A(6)=0.05,F(6)=0.3$.}
 \label{fig:b}
\end{center}\end{figure}
\begin{figure}[ht]
 \begin{center}
   \includegraphics[width=160mm]{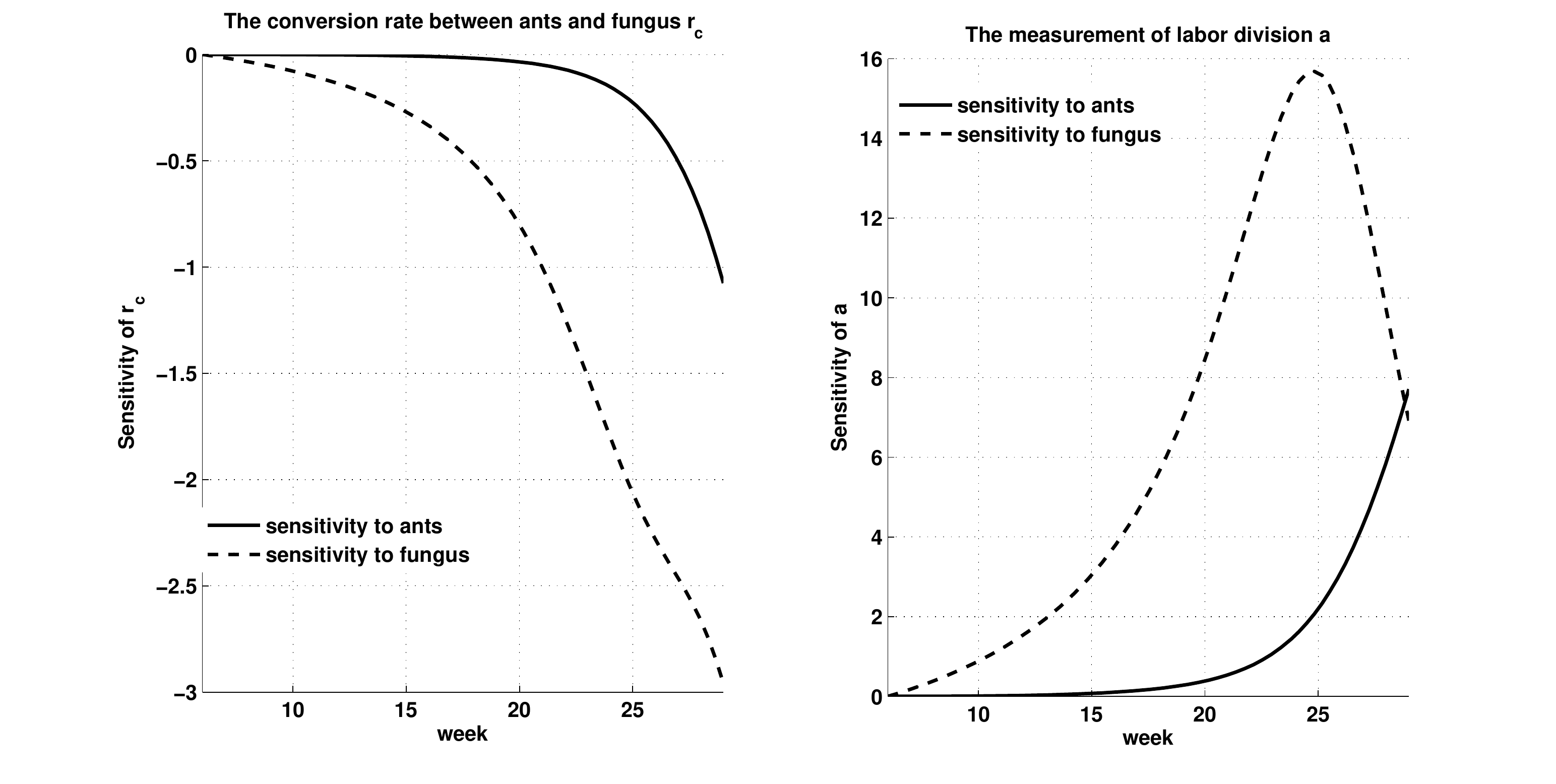}
\caption{The left figure is the sensitivity of the conversion rate between ants and fungus $r_c$ and the right figure is the sensitivity of the measure of the division of labor of ants $a$ for the model \eqref{pa1}-\eqref{pf1} around the chosen values where $r_a=0.1,r_f=.7,d_a=0.1,d_f=0.2,b=0.002,r_c=.0045, a=0.2 ,A(6)=0.05,F(6)=0.3$.}
 \label{fig:rc_a}
\end{center}\end{figure}
\begin{figure}[ht]
 \begin{center}
   \includegraphics[width=160mm]{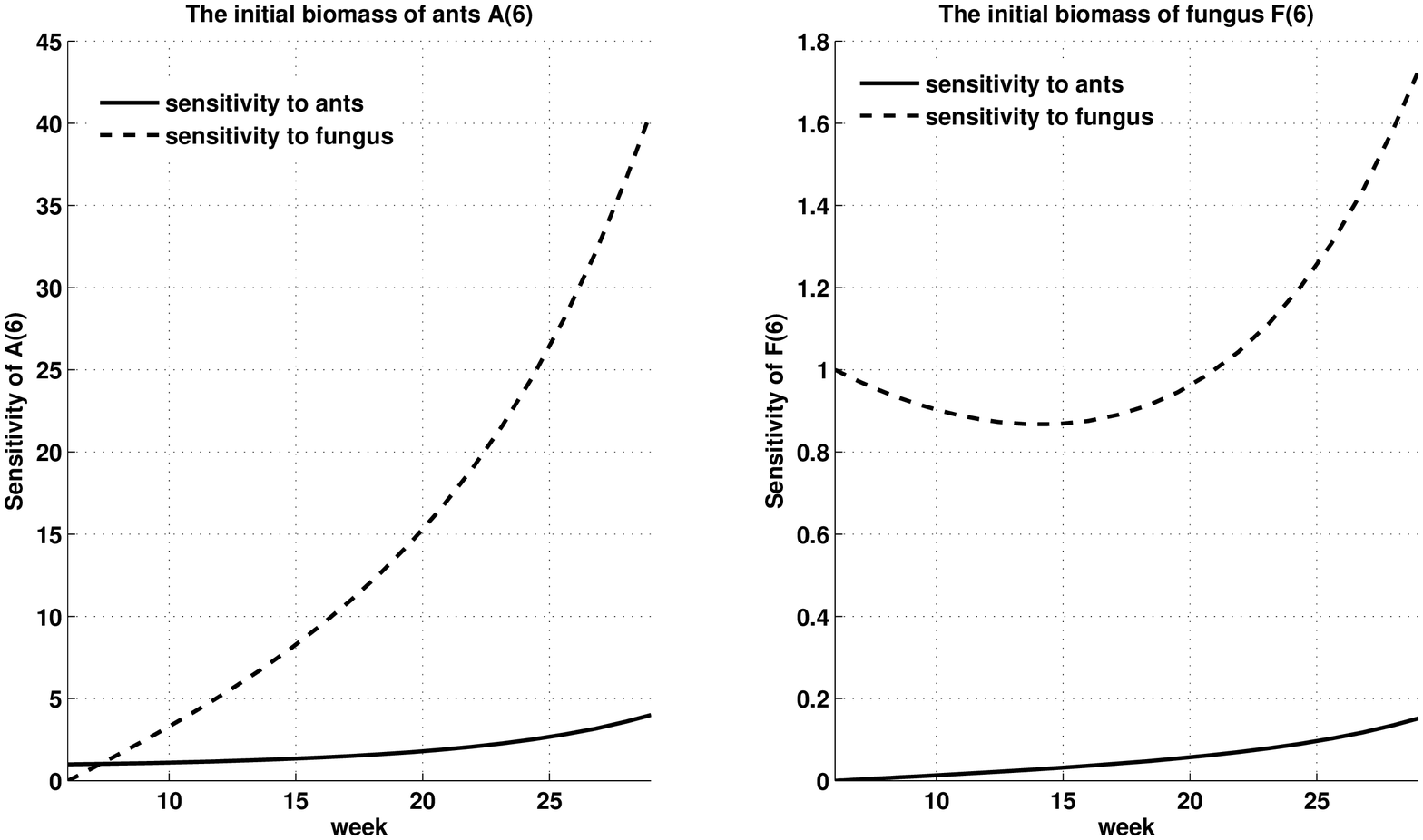}
\caption{The left figure is the sensitivity of the biomass of ants in week 6 $A(6)$ and the right figure is the sensitivity of the maximum growth rate of fungus $r_f$ for the model \eqref{pa1}-\eqref{pf1} around the chosen values where $r_a=0.1,r_f=.7,d_a=0.1,d_f=0.2,b=0.002,r_c=.0045, a=0.2 ,A(6)=0.05,F(6)=0.3$.}
 \label{fig:A0_F0}
\end{center}\end{figure}

 As we would expect, the effects of all the parameters are strictly cumulative. The effects of each parameter at the beginning of the colony are relatively small. As time progresses the parameters have a much greater effect on the model, most of them having their largest effect on the biomass of fungus at around week 25. There is a time shift between the effect on the biomass of ants and the biomass of fungus: the largest effect on the biomass of fungus is always earlier than the largest effect on the biomass of ant. In addition, all the parameters have larger effects on the biomass of fungus than on the biomass of ants. The detailed results on sensitivity analysis for each parameter and the initial condition can be summarized as follows:
\begin{enumerate}
\item Figure \ref{fig:ra_rf} shows that the growth parameters $r_a$ and $r_f$ have a positive effect on the model as a function of time; as the model progresses, $r_a$ and $r_f$ have a larger effect on the increase of both biomass of ants and fungus. The largest effect of both $r_a$ and $r_f$ on biomass of fungus occurs at week {25,} while the effect of  both $r_a$ and $r_f$ on biomass of {ants} is an increasing function of time and achieves the largest effect at the end of experiment, i.e., week 29.
\item Figure \ref{fig:da_df} shows that the death parameters $d_a$ and $d_f$ have a negative effect on the model as a function of time. As the model progresses forward in time the decrease in the biomass of ants and fungus respectively is increasingly affected by parameters $d_a$ and $d_f$. The largest effect of $r_a$ and $r_f$ on biomass of fungus occurs at week 25 and week 26 respectively while the effect of both $d_a$ and $d_f$ on biomass of ants is an increasing function of time and achieves the largest effect at the end of experiment, i.e., week 29.

\item Figure \ref{fig:rc_a} shows the conversion rate between ants and fungus $r_c$ has a negative effect on the output of the model while the measurement of the division of labor of ants has a positive effect. The effect of $r_c$ on the biomass of both ants and fungus is a decreasing function of time while the division of labor parameter $a$ shows very similar behavior to the growth parameters $r_a$ and $r_f$; this suggests that by maximizing the {efficiency of division of labor}, both the ants' and fungus's population growth will be maximized. This agrees with article main purpose number 2. 

\item Figure \ref{fig:A0_F0} shows that both initial condition $A(6)$ and $F(6)$ have positive effects on the model as a function of time. The effect of $A(6)$ on biomass of both ants and fungus is an increasing function of time. The effect of $F(6)$ on the biomass of ants is an increasing function of time while the effect on the biomass of fungus is decreasing until week 15 and then increasing until week 29.

\item Figure \ref{fig:b} shows that the half-saturation constant $b$ has a negative effect on the model as a function of time. As the model progresses forward in {time}, the decrease in biomass of ants and fungus respectively is increasingly affected by $b$. The largest effect of $b$ on fungal biomass occurs at week 25 while the effect of $b$ on biomass of {ants is} an increasing function of time and achieves the largest effect at the end of experiment, i.e., week 29. Notice that $b$ has the largest sensitivity among all the parameters and initial {conditions.}
\end{enumerate}

In conclusion, the growth parameters $r_a, r_f$, the division of labor parameter $a$ and the initial conditions $A(6), F(6)$ have positive effects on the biomass of ants and fungus while the death parameters $d_a, d_f$, the conversion rate $r_c$ and the half-saturation constant $b$ have negative effects on the biomass of ants and fungus. Among all these parameters and the initial conditions, the parameter $b$ has the largest sensitivity and the conversion rate $r_c$, the death rate of ants $d_a$ and the initial value of fungus $F(6)$ have relative small sensitivity to the output of the model.





\subsection{Parameter estimations}
In this subsection, we use Nonlinear grey-box models (System Identification Toolbox provided in MATLAB) to perform parameter estimations based on experimental data by fixing the estimated interval of $a$ to be $(0, 0.25)$ and the estimated intervals of other parameters to be $(0,\infty)$. The estimated values can be varied depending on the initial guesses and the estimated intervals. Table \ref{tab3:v_parameters} lists the best estimated values and their standard deviations when the initial guessing are the values that generated {Figures} \ref{fig:AFfit} and \ref{fig:AFration}. Table \ref{tab4:v_parameters} lists the best estimated values and their standard deviations when the initial guesses are different from the one's listed in Table \ref{tab3:v_parameters}.

\begin{table}[ht]
\begin{center}
\caption{\upshape{Parameter estimations of the system \eqref{pa1}-\eqref{pf1} case one}}\label{tab3:v_parameters}\vspace{15pt}
\begin{tabular}{|c|c|c|c|c|}\hline
 Parameters& Initial values&Intervals&Estimated values & Standard Deviation\\\hline
 $r_a$&0.1&$(0, \infty)$&0.0436&0.0949\\\hline
 $r_f$&0.7&$(0, \infty)$&1.1667&4.3344\\\hline
 $r_c$&0.0045&$(0, \infty)$&4.9781e-12&3.5671 \\\hline
 $d_a$&0.1&$(0, \infty)$&1.5144e-07&0.6700 \\\hline
 $d_f$&0.2&$(0, \infty)$&0.2707&0.3035\\\hline
  $b$&0.002&$(0, \infty)$&0.0050&34.2259\\\hline
  $a$&0.2& $(0, 0.25)$&0.1656&1140.99\\\hline
    $A(6)$&0.05& $(0, \infty)$&0.0988&DNE\\\hline
      $F(6)$&0.3&$(0, \infty)$&0.1860&DNE\\\hline
\end{tabular}
\end{center}
\end{table}

\begin{table}[ht]
\begin{center}
\caption{\upshape{Parameter estimations of the system \eqref{pa1}-\eqref{pf1} case two}}\label{tab4:v_parameters}\vspace{15pt}
\begin{tabular}{|c|c|c|c|c|}\hline
 Parameters& Initial values&Intervals&Estimated values & Standard Deviation\\\hline
 $r_a$&0.075&$(0, \infty)$&0.0747&0.0230\\\hline
 $r_f$&0.15&$(0, \infty)$&0.1585&0.2055\\\hline
 $r_c$&0.0001&$(0, \infty)$&9.9447e-05& 0.3906\\\hline
 $d_a$&0.0001&$(0, \infty)$&3.1570e-07&0.5807 \\\hline
 $d_f$&0.03&$(0, \infty)$&0.0292&0.1061\\\hline
  $b$&0.00001&$(0, \infty)$&8.2187e-05&33.6161\\\hline
  $a$&0.2& $(0, 0.25)$&0.2004&81819.3\\\hline
    $A(6)$&0.05& $(0, \infty)$&0.0498&DNE\\\hline
      $F(6)$&0.3&$(0, \infty)$&0.2995&DNE\\\hline
\end{tabular}
\end{center}
\end{table}

The comparison between Table \ref{tab3:v_parameters} and Table \ref{tab4:v_parameters} can be summarized as follows:
\begin{enumerate}
\item The different initial guesses values will give different estimated values. This may be caused by the fact that the model \eqref{pa1}-\eqref{pf1} has multiple attractors. 
\item  The estimated values of $r_c$ and $d_a$ are both extremely small. In addition, the smaller the initial guesses of $d_a$ and $r_c$, the smaller the estimated values of these parameters. This may {suggest} that $r_c$ and $d_a$ have little effect on the population dynamics of ants and fungus at the early colony stage, which has been confirmed by their small sensitivity (see Figure \ref{fig:da_df} and \ref{fig:rc_a}).
\item   The standard deviations of estimated $a,b,r_c, d_a$ are extremely large, which may be caused by two factors: 1. These parameters are not independent; 2. The extremely small value of $d_a$ and $r_c$. 
\end{enumerate}
The summary above indicates that the population dynamics of ants and fungus may be highly unstable at the early stage of colony development. Notice that collected data is from the successful colonies only. The extremely small estimated value of $r_c$ and $d_a$ may suggest that the conversion rate between ants and fungus and the death rate of ants are not as important as other factors such as the growth rate parameters $r_a, r_f$ and the death rate of fungus $d_f$.
Possibly, a multiple-stage model that includes the stages of eggs, larvae, pupae or even a stochastic model should be introduced in order to get a better understanding of the detailed ecological processes.

\section{Conclusion}\label{sec:sim}
In this article, we develop a simple mathematical model \eqref{pa1}-\eqref{pf1} to study mutualism interactions between leaf cutter ants and their fungus garden  at the early colony stage with the following unique features
\begin{enumerate}
\item The net benefit of the obligate fungus to leaf cutter ants is determined by the difference between the overall performance of collecting leaves and cultivating fungus by worker ants and the amount of fungus eaten by queen, larvae and workers; while the net benefit of obligate ants to fungus is determined by the difference between the amount of consumed fungus and the mortality rate due to the energy spent on collecting leaves and cultivating fungus. 
\item The division of labor of leaf cutter ants: workers perform different tasks to maintain their fungus gardens. This feature allows us to apply the concept of the kinetics of functional response to model the numerical functional response of fungus.
\end{enumerate}

 The mathematical analysis of \eqref{pa1}-\eqref{pf1} gives the completed global dynamics of the model (Theorem \ref{th1:extinction}, \ref{th2:localstability}, \ref{th3:basinattraction} and Corollary \ref{co:two-attractors}). These theoretical results suggests that: 1. The division of labor of ants can determine whether leaf cutter ants and their fungus garden are able to coexist; 2. When the division of labor is in a good range, the initial populations of leaf cutter ants and fungus are needed to be larger than some threshold in order to coexist.

 Finally, we validated the model \eqref{pa1}-\eqref{pf1} using empirical data. The comparison between model simulations and data supports the fact that \eqref{pa1}-\eqref{pf1} is well defined for modeling the population dynamics of the leaf cutter ants and fungus during the incipient colony stage (the early ergonomic growth stage). The good fit between the model and data also provides us an approximation of the values of difficult measured parameters such as the conversion rate between fungus and ants $r_c$ and the half-saturation constant $b$. Sensitivity analysis implies $b$ has the largest effect on the output of the model. Both sensitivity analysis and parameter estimations suggest that the population dynamics of early stage colonies may not be stable and the growth rate parameters $r_a, r_f$ and the death rate of fungus $d_f$ are important factors for determining the population dynamics for the successful colony.

The inconsistency between the data and the model fitting during week 6 to week 9 (Figure \ref{fig:AFration}) suggests that a more realistic and detailed model is needed during this period. Thus, consideration of multiple life cycle stages, including eggs, larvae and pupae, or even stochastity should be included in further modeling work. In addition, for future experiments, the biomass of, larvae and pupae should be measured as well. These different life cycle stages are likely to represent a larger proportion of the total ant biomass in early stages of colony growth and additionally likely have more variance.  This could be our future work.




\section*{Acknowledgement}

The authors would like to thank Dr. Dieter Armbruster for initiating the mathematical modeling on the interaction between leaf cutter ants and their fungus gardens.


\bibliographystyle{elsarticle-harv}

\end{document}